\documentclass[12pt,oneside]{amsart}

\usepackage{amssymb}
\usepackage{amscd}

\title[Examples of plane rational curves with two Galois points]{Examples of plane rational curves with two Galois points in positive characteristic, II} 
\author{Satoru Fukasawa and Katsushi Waki}

\address{Department of Mathematical Sciences, Faculty of Science, Yamagata University, Kojirakawa-machi 1-4-12, Yamagata 990-8560, Japan} 
\email{s.fukasawa@sci.kj.yamagata-u.ac.jp} 
\email{waki@sci.kj.yamagata-u.ac.jp}

\thanks{The first author was partially supported by JSPS KAKENHI Grant Number JP19K03438.}

\newtheorem{theorem}{Theorem}
\newtheorem{proposition}{Proposition}

\newtheorem{fact}{Fact}

\newtheorem{lemma}{Lemma}

\begin{document}
\begin{abstract}  
It is proved that there exist plane rational curves of degree twelve (resp. twenty-four) with two different outer Galois points such that the Galois group at one of two Galois points is an alternating group $A_4$ (resp. a symmetric group $S_4$) of degree four, under the assumption that the characteristic of the ground field is eleven (resp. is twenty-three). 
For an alternating group $A_5$ of degree five, a similar existence theorem is confirmed, over a field of characteristic $59$, by GAP system. 
\end{abstract}

\maketitle 

\section{Introduction} 
Let $C \subset \Bbb P^2$ be an irreducible plane curve over an algebraically closed field $k$ of characteristic $p \ge 0$ with $k(C)$ as its function field. 
For a point $P \in \Bbb P^2$, if the function field extension $k(C)/\pi_P^*k(\Bbb P^1)$ induced by the projection $\pi_P$ is Galois, then $P$ is called a Galois point for $C$. 
This notion was introduced by Yoshihara (\cite{miura-yoshihara, yoshihara1}).
Furthermore, if a Galois point $P$ is contained in $\Bbb P^2 \setminus C$, then $P$ is said to be outer. 
The associated Galois group at $P$ is denoted by $G_P$. 

It was difficult to prove the existence of plane curves $C \subset \mathbb{P}^2$ of degree $d=12$ with two Galois points $P_1$, $P_2 \in \Bbb P^2 \setminus C$ such that $G_{P_1}$ is isomorphic to an alternating group $A_4$ of degree four.  
For example, for plane rational curves $C$ of degree $d$ over a field $k$ of characteristic $p=0$, the case where $d=12, 24$ or $60$ is excluded in Yoshihara's paper \cite{yoshihara2}.  
In the present article, it is proved that there exist plane rational curves of this kind, which update the Table in \cite{yoshihara-fukasawa}, as follows. 

\begin{theorem} \label{main1}
Let $p=11$. Then there exist plane rational curves $C \subset \Bbb P^2$ of degree $d=12$ with different outer Galois points $P_1, P_2 \in \Bbb P^2 \setminus C$ such that the associated Galois groups $G_{P_1}$, $G_{P_2}$ are in the following each cases:  
\begin{itemize}
\item[(a)] $G_{P_1} \cong A_4$, $G_{P_2} \cong \Bbb{Z}/12\Bbb{Z}$;   
\item[(b)] $G_{P_1} \cong A_4$, $G_{P_2} \cong D_{12}$, where $D_{12}$ is the dihedral group of order $12$;  
\item[(c)] $G_{P_1} \cong G_{P_2} \cong A_4$.   
\end{itemize}
\end{theorem}

For a symmetric group $S_4$ of degree four, we have the following theorem. 

\begin{theorem} \label{main2}
Let $p=23$. Then there exist plane rational curves $C \subset \Bbb P^2$ of degree $d=24$ with different outer Galois points $P_1, P_2 \in \Bbb P^2 \setminus C$ such that the associated Galois groups $G_{P_1}$, $G_{P_2}$ are in the following each cases: 
\begin{itemize}
\item[(a)] $G_{P_1} \cong S_4$, $G_{P_2} \cong \Bbb{Z}/24\Bbb{Z}$;   
\item[(b)] $G_{P_1} \cong S_4$, $G_{P_2} \cong D_{24}$;  
\item[(c)] $G_{P_1} \cong G_{P_2} \cong S_4$. 
\end{itemize}
\end{theorem} 

For an alternating group $A_5$ of degree five, we have the following theorem, which will be confirmed by GAP system \cite{gap}. 

\begin{theorem} \label{main3}
Let $p=59$. Then there exist plane rational curves $C \subset \Bbb P^2$ of degree $d=60$ with different outer Galois points $P_1, P_2 \in \Bbb P^2 \setminus C$ such that the associated Galois groups $G_{P_1}$, $G_{P_2}$ are in the following each cases:  
\begin{itemize}
\item[(a)] $G_{P_1} \cong A_5$, $G_{P_2} \cong \Bbb{Z}/60\Bbb{Z}$;   
\item[(b)] $G_{P_1} \cong A_5$, $G_{P_2} \cong D_{60}$;  
\item[(c)] $G_{P_1} \cong G_{P_2} \cong A_5$. 
\end{itemize}
\end{theorem}

\section{Proof of Theorem \ref{main1}}

According to \cite[Theorem 1 and Remark 1]{fukasawa} and L\"{u}roth's theorem, the following theorem holds (see also \cite[Fact 2]{fukasawa-higashine}, \cite[Problem 1]{fukasawa-waki}). 

\begin{fact} \label{criterion1}
Let $Q \in \Bbb P^1$ and let $G_1, G_2 \subset PGL(2, k)$ be different finite subgroups with $|G_1|=|G_2|$. 
If the two conditions
\begin{itemize}
\item[(a)] $G_1 \cap G_2=\{1\}$, and
\item[(b)] the orbits $G_1Q, G_2Q$ have length $|G_1|=|G_2|$ and $G_1Q=G_2Q$
\end{itemize} 
are satisfied, then there exists a plane rational curve $C \subset \Bbb{P}^2$ of degree $|G_1|=|G_2|$ with two different outer Galois points $P_1, P_2 \in \Bbb P^2 \setminus C$ such that $G_{P_i}=G_i$ for $i=1, 2$.  
\end{fact}

To prove Theorem \ref{main1} (a) (resp. Theorem \ref{main1} (b), Theorem \ref{main1} (c)), we have to prove the existence of a pair  $(G_1, G_2)$ of subgroups $G_1, G_2 \subset PGL(2, k)$ such that $G_1 \cong A_4$ and $G_2 \cong \Bbb Z/12\Bbb Z$ (resp. $G_1 \cong A_4$ and $G_2 \cong D_{12}$, $G_1 \cong G_2 \cong A_4$), and conditions (a) and (b) in Fact \ref{criterion1} are satisfied for $(G_1, G_2)$.  

Let $\alpha=2 \in \mathbb{F}_{11}$, which is a primitive element. 
Let $\sigma, \tau, \eta \in {\rm Aut}(\Bbb P^1) \cong PGL(2, k)$ be represented by matrices
$$ 
A_{\sigma}=\left ( \begin{array}{cc}
0 & \alpha \\
1 & 0 
\end{array}\right), \ 
A_{\tau}=
\left (\begin{array}{cc}
1 & \alpha \\
-1 & -1 
\end{array}\right), \
A_{\eta}=\left (
\begin{array}{cc}
\alpha & \alpha^4 \\
1 & \alpha^2 
\end{array} \right)
$$
respectively, that is, $\sigma(s, t)=(s, t)A_{\sigma}$, $\tau(s, t)=(s, t)A_{\tau}$, $\eta(s, t)=(s, t)A_{\eta}$. 
The following lemma can be confirmed by hand. 

\begin{lemma} \label{calculation1}
For matrices $A, B \in GL(2, k)$, if the classes of $A, B$ are the same, then it is denoted by $A \sim B$. 
Then: 
\begin{itemize}
\item[(a)] $\sigma$, $\tau$ are of order two. 
\item[(b)] $\eta$ is of order three. 
\item[(c)] $A_{\sigma}A_{\tau} \sim A_{\tau}A_{\sigma}$. 
\item[(d)] $A_{\eta}^{-1}A_{\sigma}A_{\eta} \sim A_{\tau}$. 
\item[(e)] $A_{\eta}^{-1}A_{\tau}A_{\eta} \sim A_{\sigma}A_{\tau}$.
\item[(f)] The group $G_1:=\langle \sigma, \tau, \eta \rangle$ is isomorphic to $A_4$. 
\item[(g)] The following three matrices represent all elements of $G_1 \cong A_4$ of order two: 
$$ 
\left ( \begin{array}{cc}
0 & 2 \\
1 & 0 
\end{array}\right), \ 
\left (\begin{array}{cc}
10 & 9 \\
1 & 1 
\end{array}\right), \
\left (
\begin{array}{cc}
9 & 9 \\
1 & 2
\end{array} \right). 
$$
\item[(h)] The following eight matrices represent all elements of $G_1 \cong A_4$ of order three: 
\begin{eqnarray*} 
& & 
\left ( \begin{array}{cc}
2 & 5 \\
1 & 4 
\end{array}\right), \ 
\left (\begin{array}{cc}
7 & 5 \\
1 & 9 
\end{array}\right), \
\left (
\begin{array}{cc}
4 & 1 \\
1 & 6
\end{array} \right),  
\left ( \begin{array}{cc}
3 & 4 \\
1 & 10 
\end{array}\right), \\
& & 
\left (\begin{array}{cc}
1 & 4 \\
1 & 8 
\end{array}\right), \
\left (
\begin{array}{cc}
8 & 3 \\
1 & 5
\end{array} \right),  
\left ( \begin{array}{cc}
6 & 3 \\
1 & 3 
\end{array}\right), \ 
\left (\begin{array}{cc}
5 & 4 \\
1 & 7 
\end{array}\right). 
\end{eqnarray*}
\item[(i)] The group $G_1$ acts on the set $\Bbb P^1(\Bbb F_{11})$ transitively, where $\Bbb P^1(\Bbb F_{11})$ is the set of all $\Bbb F_{11}$-rational points of $\Bbb P^1$. 
\end{itemize}
\end{lemma}

Let $\xi$ be an automorphism of $\Bbb P^1$ represented by  
$$ 
A_{\xi}=
\left ( \begin{array}{cc}
\alpha & 1 \\
1 & 0 
\end{array} \right ), 
$$
and let $G_2:=\langle \xi \rangle$. 
It is not difficult to confirm the following lemma. 

\begin{lemma} \label{calculation2}
The order of $\xi$ is twelve, and $G_2$ acts on $\Bbb P^1(\Bbb F_{11})$ transitively. 
\end{lemma}

For the following proposition, we give the proof of assertion (a), which is the most complicated part of the proof of Theorem \ref{main1} (a).  

\begin{proposition} \label{calculation3}
The pair $(G_1, G_2)$ satisfies the following two conditions:  
\begin{itemize}
\item[(a)] $G_1 \cap G_2=\{1\}$. 
\item[(b)] Let $Q \in \Bbb P^1(\Bbb F_{11})$. 
Then $G_1Q=\Bbb P^1(\Bbb F_{11})=G_2Q$. 
\end{itemize}
In particular, conditions (a) and (b) in Fact \ref{criterion1} are satisfied for the pair $(G_1, G_2)$.  
\end{proposition}

\begin{proof}
We prove assertion (a). 
Assume by contradiction that $G_1 \cap G_2 \ne \{1\}$. 
Since $G_1 \cap G_2$ is a subgroup of $A_4$ and of a cyclic group $\langle \xi \rangle \cong \mathbb{Z}/12\mathbb{Z}$, it follows that $\xi^6 \in G_1$ or $\xi^4 \in G_1$. 
Note that 
$$ A_{\xi}^{6} \sim \left ( \begin{array}{cc}
1 & 1 \\
1 & 10 
\end{array} \right ), \
A_{\xi}^{4} \sim \left ( \begin{array}{cc}
7 & 1 \\
1 & 5 
\end{array} \right ). $$
According to Lemma \ref{calculation1} (g) (resp. Lemma \ref{calculation1} (h)), it follows that $\xi^6 \not\in G_1$ (resp. that $\xi^4 \not\in G_1$). 
This is a contradiction. 
\end{proof}

\begin{proof}[Proof of Theorem \ref{main1} (a)]
Proposition \ref{calculation3} and Fact \ref{criterion1} imply Theorem \ref{main1} (a).  
\end{proof} 

Let $\sigma', \tau' \in {\rm Aut}(\Bbb P^1)$ be represented by matrices
$$ 
A_{\sigma'}=\left ( \begin{array}{cc}
0 & \alpha^3 \\
1 & 0 
\end{array}\right), \ 
A_{\tau'}=
\left (\begin{array}{cc}
\alpha^2 & 1 \\
\alpha^2 & \alpha^4 
\end{array}\right)
$$
respectively. 
It is not difficult to confirm the following lemma. 

\begin{lemma} \label{calculation4}
\begin{itemize}
\item[(a)] $\sigma'$ is of order two. 
\item[(b)] $\tau'$ is of order six. 
\item[(c)] $A_{\sigma'}^{-1}A_{\tau'}A_{\sigma'} \sim A_{\tau'}^{-1}$. 
\item[(d)] The group $G_3:=\langle \sigma', \tau' \rangle$ is isomorphic to $D_{12}$.
\item[(e)] The group $G_3$ acts on $\Bbb P^1(\Bbb F_{11})$ transitively. 
\end{itemize}
\end{lemma}

For the following proposition, we give the proof of assertion (a), which is the most complicated part of the proof of Theorem \ref{main1} (b).

\begin{proposition} \label{calculation5}
The pair $(G_1, G_3)$ satisfies the following two conditions:  
\begin{itemize}
\item[(a)] $G_1 \cap G_3=\{1\}$. 
\item[(b)] Let $Q \in \Bbb P^1(\Bbb F_{11})$. 
Then $G_1Q=\Bbb P^1(\Bbb F_{11})=G_3Q$. 
\end{itemize}
In particular, conditions (a) and (b) in Fact \ref{criterion1} are satisfied for the pair $(G_1, G_3)$.  
\end{proposition}

\begin{proof}
We prove assertion (a). 
Assume by contradiction that $G_1 \cap G_3 \ne \{1\}$. 
Since $G_1 \cap G_3$ is a subgroup of $A_4$, it follows that $G_1 \cap G_3$ contains an element of order two or three.

Assume that $G_1 \cap G_3$ contains an element of order three. 
Since $G_1 \cap G_3$ is a subgroup of $G_3 \cong D_{12}$, it follows that $\tau'^2 \in G_1$. 
Note that 
$$ 
A_{\tau'}^{2} \sim \left ( \begin{array}{cc}
3 & 3 \\
1 & 6 
\end{array} \right ) 
$$
According to Lemma \ref{calculation1} (h), it follows that $\tau'^2 \not\in G_1$. 
This is a contradiction. 

Assume that $G_1 \cap G_3$ contains an element of order two. 
It follows that $\sigma \in G_1 \cap G_3$, $\tau \in G_1 \cap G_3$ or $\sigma\tau \in G_1 \cap G_3$. 
Assume that $\sigma \in G_1 \cap G_3$. 
Then $\sigma\tau'^3=\tau'^3\sigma$.  
Note that 
$$ 
A_{\tau'}^{3 }\sim \left ( \begin{array}{cc}
4 & 3 \\
1 & 7 
\end{array} \right ) 
$$
It follows that
$$ 
A_{\sigma}A_{\tau'}^{3} \sim \left ( \begin{array}{cc}
6 & 9 \\
1 & 9 
\end{array} \right ),  \
A_{\tau'}^{3}A_{\sigma} \sim \left ( \begin{array}{cc}
2 & 9 \\
1 & 5 
\end{array} \right ). 
$$
This implies that $\sigma\tau'^3 \ne \tau'^3\sigma$. 
This is a contradiction. 

Assume that $\tau \in G_1 \cap G_3$. 
Then $\tau\tau'^3=\tau'^3\tau$.  
It follows that
$$ 
A_{\tau}A_{\tau'}^{3} \sim \left ( \begin{array}{cc}
1 & 1 \\
1 & 2 
\end{array} \right ),  \
A_{\tau'}^{3}A_{\tau} \sim \left ( \begin{array}{cc}
9 & 1 \\
1 & 10 
\end{array} \right ). 
$$
This implies that $\tau\tau'^3 \ne \tau'^3\tau$. 
This is a contradiction. 

Assume that $\sigma\tau \in G_1 \cap G_3$. 
Then $(\sigma\tau)\tau'^3=\tau'^3(\sigma\tau)$.  
It follows that
$$ 
A_{\sigma\tau}A_{\tau'}^{3} \sim \left ( \begin{array}{cc}
2 & 4 \\
1 & 1 
\end{array} \right ),  \
A_{\tau'}^{3}A_{\sigma\tau} \sim \left ( \begin{array}{cc}
10 & 4 \\
1 & 9 
\end{array} \right ). 
$$
This implies that $(\sigma\tau)\tau'^3 \ne \tau'^3(\sigma\tau)$. 
This is a contradiction. 
\end{proof}

\begin{proof}[Proof of Theorem \ref{main1} (b)]
Proposition \ref{calculation5} and Fact \ref{criterion1} imply Theorem \ref{main1} (b).  
\end{proof} 

Let $\iota$ be an automorphism of $\Bbb P^1$ represented by  
$$ 
A_{\iota}=
\left ( \begin{array}{cc}
\alpha & 0 \\
0 & 1 
\end{array} \right ), 
$$
and let $G_4:=\iota G_1 \iota^{-1}$. 
For the following proposition, we give the proof of assertion (a), which is the most complicated part of the proof of Theorem \ref{main1} (c).

\begin{proposition} \label{calculation6}
The pair $(G_1, G_4)$ satisfies the following two conditions:  
\begin{itemize}
\item[(a)] $G_1 \cap G_4=\{1\}$. 
\item[(b)] Let $Q \in \Bbb P^1(\Bbb F_{11})$. 
Then $G_1Q=\Bbb P^1(\Bbb F_{11})=G_4Q$. 
\end{itemize}
In particular, conditions (a) and (b) in Fact \ref{criterion1} are satisfied for the pair $(G_1, G_4)$.  
\end{proposition}

\begin{proof}
We prove assertion (a). 
The following three matrices represent all elements of $G_4$ of order two: 
$$ 
\left ( \begin{array}{cc}
0 & 6 \\
1 & 0 
\end{array}\right), \ 
\left (\begin{array}{cc}
5 & 5 \\
1 & 6 
\end{array}\right), \
\left (
\begin{array}{cc}
10 & 5 \\
1 & 1
\end{array} \right). 
$$
According to Lemma \ref{calculation1} (g), it follows that $G_1 \cap G_4$ does not contain an element of order two. 
The following eight matrices represent all elements of $G_4$ of order three: 
\begin{eqnarray*} 
& & 
\left ( \begin{array}{cc}
1 & 4 \\
1 & 2 
\end{array}\right), \ 
\left (\begin{array}{cc}
9 & 4 \\
1 & 10 
\end{array}\right), \
\left (
\begin{array}{cc}
2 & 3 \\
1 & 3
\end{array} \right),  
\left ( \begin{array}{cc}
7 & 1 \\
1 & 5 
\end{array}\right), \\
& &
\left (\begin{array}{cc}
6 & 2 \\
1 & 4 
\end{array}\right), \
\left (
\begin{array}{cc}
4 & 9 \\
1 & 8
\end{array} \right),  
\left ( \begin{array}{cc}
3 & 9 \\
1 & 7 
\end{array}\right), \ 
\left (\begin{array}{cc}
8 & 3 \\
1 & 9 
\end{array}\right). 
\end{eqnarray*} 
According to Lemma \ref{calculation1} (h), it follows that $G_1 \cap G_4$ does not contain an element of order three. 
The claim follows. 
\end{proof}

\begin{proof}[Proof of Theorem \ref{main1} (c)]
Proposition \ref{calculation6} and Fact \ref{criterion1} imply Theorem \ref{main1} (c).  
\end{proof}

\section{Proof of Theorem \ref{main2}}

Let $\alpha=5 \in \mathbb{F}_{23}$, which is a primitive element. 
Let $\sigma, \tau, \eta, \mu:=\eta^2 \in {\rm Aut}(\Bbb P^1) \cong PGL(2, k)$ be represented by matrices
$$ 
A_{\sigma}=\left ( \begin{array}{cc}
0 & 1 \\
\alpha^7 & 0 
\end{array}\right), \ 
A_{\tau}=
\left (\begin{array}{cc}
\alpha^{12} & \alpha^7 \\
1 & \alpha^3 
\end{array}\right), \
A_{\eta}=\left (
\begin{array}{cc}
1 & \alpha^{10} \\
\alpha^6 & \alpha^{15} 
\end{array} \right), \ 
A_{\mu}=\left (
\begin{array}{cc}
-1 & 1 \\
-\alpha^7 & 1 
\end{array} \right)
$$
respectively. 
The following lemma can be confirmed by hand. 

\begin{lemma} \label{calculation1'}
\begin{itemize}
\item[(a)] $\sigma$, $\mu$ are of order two. 
\item[(b)] $\tau$ is of order three. 
\item[(c)] $\eta$ is of order four. 
\item[(d)] $A_{\sigma}A_{\mu} \sim A_{\mu}A_{\sigma}$. 
\item[(e)] $A_{\tau}^{-1}A_{\sigma}A_{\tau} \sim A_{\mu}$, $A_{\tau}^{-1}A_{\mu}A_{\tau} \sim A_{\sigma}A_{\mu}$. 
\item[(f)] $A_{\eta}^{-1}A_{\sigma}A_{\eta} \sim A_{\sigma}A_{\mu}$, $A_{\eta}^{-1}A_{\mu}A_{\eta} \sim A_{\mu}$.
\item[(g)] $A_{\eta}^{-1}A_{\tau}A_{\eta} \sim A_{\sigma}A_{\mu}A_{\tau}^2$. 
\item[(h)] $A_{\eta}^{-1}A_{\tau}^2A_{\eta} \sim A_{\mu}A_{\tau}$
\item[(i)] $\langle \sigma, \mu, \tau \rangle \cong A_4$, and $G_1:=\langle \sigma, \tau, \eta, \mu \rangle \cong S_4$.  
\item[(j)] The group $G_1$ acts on the set $\Bbb P^1(\Bbb F_{23})$ transitively. 
\item[(k)] Let 
\begin{eqnarray*}
O_1&=&\{(0:1), (1:\alpha), (1:\alpha^3), (1:\alpha^6), (1:\alpha^7), (1:\alpha^{18})\}, \\
O_2&=&\{(1:0), (1:\alpha^8), (1:\alpha^9), (1:\alpha^{12}), (1:\alpha^{14}), (1:\alpha^{19}) \}, \\
O_3&=&\{(1:1), (1:\alpha^2), (1:\alpha^4), (1:\alpha^{10}), (1:\alpha^{17}), (1:\alpha^{21})\}, \\
O_4&=&\{(1:\alpha^5), (1:\alpha^{11}), (1:\alpha^{13}), (1:\alpha^{15}), (1:\alpha^{16}), (1:\alpha^{20}) \}, 
\end{eqnarray*}
namely, $\mathbb{P}^1(\mathbb{F}_{23})=O_1 \cup O_2 \cup O_3 \cup O_4$. 
Then it follows that 
\begin{eqnarray*}
& & \sigma(O_1)=O_2, \sigma(O_2)=O_1, \sigma(O_3)=O_4, \sigma(O_4)=O_3, \\
& & \tau(O_1)=O_1, \tau(O_2)=O_3, \tau(O_3)=O_4, \tau(O_4)=O_2, \\
& & \eta(O_1)=O_2, \eta(O_2)=O_3, \eta(O_3)=O_4, \eta(O_4)=O_1. 
\end{eqnarray*} 
In particular, $G_1$ acts on the set $\{O_1, O_2, O_3, O_4\}$ faithfully. 
\end{itemize}
\end{lemma}

Let $\xi$ be an automorphism of $\Bbb P^1$ represented by  
$$ 
A_{\xi}=
\left ( \begin{array}{cc}
0 & -1 \\
-1 & 1 
\end{array} \right ), 
$$
and let $G_2:=\langle \xi \rangle$. 
It is not difficult to confirm the following lemma. 

\begin{lemma} \label{calculation2'}
The order of $\xi$ is twenty-four, and $G_2$ acts on $\Bbb P^1(\Bbb F_{23})$ transitively. 
\end{lemma}

For the following proposition, we give the proof of assertion (a), which is the most complicated part of the proof of Theorem \ref{main2} (a).

\begin{proposition} \label{calculation3'}
The pair $(G_1, G_2)$ satisfies the following two conditions:  
\begin{itemize}
\item[(a)] $G_1 \cap G_2=\{1\}$. 
\item[(b)] Let $Q \in \Bbb P^1(\Bbb F_{23})$. 
Then $G_1Q=\Bbb P^1(\Bbb F_{23})=G_2Q$. 
\end{itemize}
In particular, conditions (a) and (b) in Fact \ref{criterion1} are satisfied for the pair $(G_1, G_2)$.  
\end{proposition}

\begin{proof}
We prove assertion (a). 
Assume by contradiction that $G_1 \cap G_2 \ne \{1\}$. 
Since $G_1 \cap G_2$ is a subgroup of $S_4$ and of a cyclic group $\langle \xi \rangle \cong \mathbb{Z}/24\mathbb{Z}$, it follows that $\xi^{12} \in G_1$ or $\xi^{8} \in G_1$. 
Note that 
$$ A_{\xi}^{12} \sim \left ( \begin{array}{cc}
-1 & -3 \\
-3 & 1 
\end{array} \right ), \
A_{\xi}^{8} \sim \left ( \begin{array}{cc}
13 & 2 \\
2 & 11 
\end{array} \right ). $$
It follows that 
$$ \xi^{12}(0:1)=(1:\alpha^{17}) \in O_3 \ \mbox{ and } \ \xi^8(0:1)=(1:\alpha^{7}) \in O_1 $$
for the point $(0:1) \in O_1$. 
According to Lemma \ref{calculation1'} (k), it follows that $\xi^{12}(1:\alpha) \in O_3$ and $\xi^8(1:\alpha^3) \in O_1$, since $(1:\alpha), (1:\alpha^3) \in O_1$. 
However, 
$$ \xi^{12}(1:\alpha)=(1:\alpha^5) \in O_4, \ \xi^{8}(1:\alpha^3)=(1:\alpha^{16}) \in O_4.$$
This is a contradiction. 
\end{proof}

\begin{proof}[Proof of Theorem \ref{main2} (a)]
Proposition \ref{calculation3'} and Fact \ref{criterion1} imply Theorem \ref{main2} (a).  
\end{proof} 

Let $\sigma', \tau' \in {\rm Aut}(\Bbb P^1)$ be represented by matrices
$$ 
A_{\sigma'}=\left ( \begin{array}{cc}
0 & \alpha^{10} \\
\alpha^9 & 0 
\end{array}\right), \ 
A_{\tau'}=
\left (\begin{array}{cc}
\alpha^{15} & \alpha \\
-1 & \alpha^7 
\end{array}\right)
$$
respectively. 
The following lemma can be confirmed by hand.

\begin{lemma} \label{calculation4'}
\begin{itemize}
\item[(a)] $\sigma'$ is of order two. 
\item[(b)] $\tau'$ is of order twelve. 
\item[(c)] $A_{\sigma'}^{-1}A_{\tau'}A_{\sigma'} \sim A_{\tau'}^{-1}$. 
\item[(d)] The group $G_3:=\langle \sigma', \tau' \rangle$ is isomorphic to $D_{24}$.
\item[(e)] The group $G_3$ acts on $\Bbb P^1(\Bbb F_{23})$ transitively. 
\item[(f)] Let 
\begin{eqnarray*}
T_1 &=& \{ (0:1), (1:\alpha^{18}), (1:\alpha^3), (1:\alpha^{11}), (1:\alpha^4), (1:\alpha^9), \\ 
& & \ \ (1:1), (1:\alpha^{21}), (1:\alpha^{13}), (1:\alpha^{16}), (1:\alpha^{17}), (1:\alpha^{15}) \}, \\
T_2 &=& \{ (1:0), (1:\alpha^8), (1:\alpha^6), (1:\alpha^7), (1:\alpha^{10}), (1:\alpha^2),  \\ 
& & \ \ (1:\alpha), (1:\alpha^{14}), (1:\alpha^{19}), (1:\alpha^{12}), (1:\alpha^{20}), (1:\alpha^5) \}, 
\end{eqnarray*}
namely, $\mathbb{P}^1(\mathbb{F}_{23})=T_1 \cup T_2$. 
Then it follows that 
$$ \sigma'(T_1)=T_2, \sigma'(T_2)=T_1, \tau'(T_1)=T_1, \tau'(T_2)=T_2. $$
In particular, $G_3$ acts on the set $\{T_1, T_2\}$. 
\end{itemize}
\end{lemma}

For the following proposition, we give the proof of assertion (a), which is the most complicated part of the proof of Theorem \ref{main2} (b). 

\begin{proposition} \label{calculation5'}
The pair $(G_1, G_3)$ satisfies the following two conditions:  
\begin{itemize}
\item[(a)] $G_1 \cap G_3=\{1\}$. 
\item[(b)] Let $Q \in \Bbb P^1(\Bbb F_{23})$. 
Then $G_1Q=\Bbb P^1(\Bbb F_{23})=G_3Q$. 
\end{itemize}
In particular, conditions (a) and (b) in Fact \ref{criterion1} are satisfied for the pair $(G_1, G_3)$.  
\end{proposition}

\begin{proof}
We prove assertion (a). 
It follows from Lemmas \ref{calculation1'} (k) and \ref{calculation4'} (f) that
\begin{eqnarray*}
& & O_1 \cap T_1=\{(0:1), (1:\alpha^3), (1:\alpha^{18})\}, 
O_1 \cap T_2=\{(1:\alpha), (1:\alpha^6), (1:\alpha^7) \}, \\ 
& & O_2 \cap T_1=\{(1:\alpha^9)\},
O_2 \cap T_2=\{(1:0), (1:\alpha^8), (1:\alpha^{12}), (1:\alpha^{14}), (1:\alpha^{19})\}, \\ 
& & O_3 \cap T_1=\{(1:1), (1:\alpha^4), (1:\alpha^{17}), (1:\alpha^{21})\}, O_3 \cap T_2=\{(1:\alpha^2), (1:\alpha^{10})\}, \\
& & O_4 \cap T_1=\{(1:\alpha^5), (1:\alpha^{20})\}, O_4 \cap T_2=\{(1:\alpha^{11}), (1:\alpha^{13}), (1:\alpha^{15}), (1:\alpha^{16}) \}.  
\end{eqnarray*}
Note that if $O_i \cap T_j$ consists of a unique point, then $i=2$ and $j=1$. 
Let $\gamma \in G_1 \cap G_3$. 
Then $\gamma(1:\alpha^9)=(1:\alpha^9)$. 
Since $G_3$ acts on $\mathbb{P}(\mathbb{F}_{23})$ transitively, it follows that $\gamma=1$. 
\end{proof}

\begin{proof}[Proof of Theorem \ref{main2} (b)]
Proposition \ref{calculation5'} and Fact \ref{criterion1} imply Theorem \ref{main2} (b).  
\end{proof} 

Let $\iota$ be an automorphism of $\Bbb P^1$ represented by  
$$ 
A_{\iota}=
\left ( \begin{array}{cc}
\alpha^7 & 0 \\
0 & 1 
\end{array} \right ), 
$$
and let $G_4:=\iota G_1 \iota^{-1}$. 
For the following proposition, we give the proof of assertion (a), which is the most complicated part of the proof of Theorem \ref{main2} (c).

\begin{proposition} \label{calculation6'}
The pair $(G_1, G_4)$ satisfies the following two conditions:  
\begin{itemize}
\item[(a)] $G_1 \cap G_4=\{1\}$. 
\item[(b)] Let $Q \in \Bbb P^1(\Bbb F_{23})$. 
Then $G_1Q=\Bbb P^1(\Bbb F_{23})=G_4Q$. 
\end{itemize}
In particular, conditions (a) and (b) in Fact \ref{criterion1} are satisfied for the pair $(G_1, G_4)$.  
\end{proposition}

\begin{proof}
We prove assertion (a). 
It follows that 
\begin{eqnarray*}
\iota(O_1)&=&\{(0:1), (1:\alpha^{16}), (1:\alpha^{18}), (1:\alpha^{21}), (1:1), (1:\alpha^{11})\}, \\
\iota(O_2)&=&\{(1:0), (1:\alpha), (1:\alpha^2), (1:\alpha^{5}), (1:\alpha^{7}), (1:\alpha^{12}) \}, \\
\iota(O_3)&=&\{(1:\alpha^{15}), (1:\alpha^{17}), (1:\alpha^{19}), (1:\alpha^{3}), (1:\alpha^{10}), (1:\alpha^{14})\}, \\
\iota(O_4)&=&\{(1:\alpha^{20}), (1:\alpha^{4}), (1:\alpha^{6}), (1:\alpha^{8}), (1:\alpha^{9}), (1:\alpha^{13}) \}.  
\end{eqnarray*}
Note that $O_i \cap \iota(O_j)=\emptyset$ implies $i=2$ and $j=1$. 
Assume by contradiction that there exists $\gamma \in G_1 \cap G_4 \setminus \{1\}$. 
According to Lemma \ref{calculation1'} (k), it follows that $\gamma(\iota(O_1))=\iota(O_1)$. 
Since $\iota(O_1) \subset O_1 \cup O_3 \cup O_4$, it follows that $\gamma$ acts on $\{O_1, O_3, O_4\}$, that is, $\gamma(O_2)=O_2$. 

Assume that $\gamma$ is of order three. 
Then $\gamma(O_1)=O_3$ or $\gamma^2(O_1)=O_3$. 
We can assume that $\gamma(O_1)=O_3$. 
Then $\gamma$ coincides with the automorphism $\sigma\tau^2\sigma \in G_1$, which is represented by the matrix
$$ 
A_{\sigma}A_{\tau}^2A_{\sigma}\sim
\left(
\begin{array}{cc}
0 & 1 \\
17 & 0
\end{array} \right)
\left(
\begin{array}{cc}
-4 & -7 \\
5 & 2
\end{array} \right)
\left(
\begin{array}{cc}
0 & 1 \\
17 & 0
\end{array} \right)
\sim
\left(
\begin{array}{cc}
11 & 5 \\
1 & 1
\end{array} \right). 
$$  
Then for points $(1:0), (1:\alpha^2) \in \iota(O_2)$, $\gamma(1:0)=(1:\alpha^{14}) \in \iota(O_3)$ and $\gamma(1:\alpha^2)=(1:\alpha^5) \in \iota(O_2)$. 
Since $\gamma \in G_4$, this is a contradiction. 

Assume that $\gamma$ is of order two. 
Then $\gamma(O_i)=O_i$ for some $i=1, 3, 4$. 
If $i=1$, then $\gamma$ coincides the automorphism $\tau^{-1}\sigma\eta\tau \in G_1$, which is represented by 
$$ A_{\tau}A_{\eta}A_{\sigma}A_{\tau}^2\sim 
\left(\begin{array}{cc}
-5 & -6 \\
1 & 10 
\end{array}\right)
\left(\begin{array}{cc}
1 & 9 \\
8 & -4 
\end{array}\right)
\left(\begin{array}{cc}
0 & 1 \\
17 & 0 
\end{array}\right) 
\left(\begin{array}{cc}
-4 & -7 \\
5 & 2 
\end{array}\right)
\sim\left(\begin{array}{cc}
-10 & 1 \\
6 & 10 
\end{array}\right).$$
Then for points $(0:1), (1:1) \in \iota(O_1)$, $\gamma(0:1)=(1:\alpha^{7}) \in \iota(O_2)$ and $\gamma(1:1)=(1:\alpha^{16}) \in \iota(O_1)$. 
Since $\gamma \in G_4$, this is a contradiction. 
If $i=4$, then $\gamma$ coincides the automorphism $\gamma_0:=\eta^2\sigma\eta\in G_1$, which is represented by 
$$ A_{\eta}A_{\sigma}A_{\eta}^2\sim
\left(\begin{array}{cc}
1 & 9 \\
8 & -4 
\end{array}\right)
\left(\begin{array}{cc}
0 & 1 \\
17 & 0 
\end{array}\right) 
\left(\begin{array}{cc}
4 & -4 \\
-1 & -4 
\end{array}\right)
\sim\left(\begin{array}{cc}
-10 & 5 \\
-4 & 10 
\end{array}\right).$$
Then for points $(0:1), (1:1) \in \iota(O_1)$, $\gamma(0:1)=(1:\alpha^{10}) \in \iota(O_3)$ and $\gamma(1:1)=(1:\alpha^{7}) \in \iota(O_2)$. 
Since $\gamma \in G_4$, this is a contradiction. 
If $i=3$, then $\gamma$ coincides the automorphism $\tau\gamma_0\tau^2$, which is represented by 
$$ A_{\tau}^2A_{\eta^2\sigma\eta}A_{\tau}\sim
\left(\begin{array}{cc}
-4 & -7 \\
5 & 2 
\end{array}\right)
\left(\begin{array}{cc}
-10 & 5 \\
-4 & 10 
\end{array}\right) 
\left(\begin{array}{cc}
-5 & -6 \\
1 & 10 
\end{array}\right)
\sim\left(\begin{array}{cc}
7 & 3 \\
-10 & -7 
\end{array}\right).$$
Then for points $(0:1), (1:1) \in \iota(O_1)$, $\gamma(0:1)=(1:\alpha^{16}) \in \iota(O_1)$ and $\gamma(1:1)=(1:\alpha^{10}) \in \iota(O_3)$. 
Since $\gamma \in G_4$, this is a contradiction. 
\end{proof}

\begin{proof}[Proof of Theorem \ref{main2} (c)]
Proposition \ref{calculation6'} and Fact \ref{criterion1} imply Theorem \ref{main2} (c).  
\end{proof} 

\section{Proof of Theorem \ref{main3}}
All lemmas and propositions in this section can be confirmed by GAP system \cite{gap}. 

Let $\alpha=2 \in \mathbb{F}_{59}$, which is a primitive element. 
Let $\sigma, \tau \in {\rm Aut}(\Bbb P^1) \cong PGL(2, k)$ be represented by matrices
$$ 
A_{\sigma}=\left ( \begin{array}{cc}
-\alpha^{26} & 1 \\
\alpha^{27} & \alpha^{26} 
\end{array}\right), \ 
A_{\tau}=
\left (\begin{array}{cc}
1 & \alpha \\
\alpha^6 & \alpha^{34} 
\end{array}\right), 
$$
respectively.

\begin{lemma} \label{calculation1''}
\begin{itemize}
\item[(a)] $\sigma$ is of order two. 
\item[(b)] $\tau$ is of order three. 
\item[(c)] The group $G_1:=\langle \sigma, \tau \rangle$ is isomorphic to $A_5$.  
\item[(d)] The group $G_1$ acts on the set $\Bbb P^1(\Bbb F_{59})$ transitively. 
\end{itemize}
\end{lemma}

Let $\xi$ be an automorphism of $\Bbb P^1$ represented by  
$$ 
A_{\xi}=
\left ( \begin{array}{cc}
1 & 1 \\
\alpha^{12} & 0 
\end{array} \right ), 
$$
and let $G_2:=\langle \xi \rangle$. 

\begin{lemma} \label{calculation2''}
The order of $\xi$ is sixty, and $G_2$ acts on $\Bbb P^1(\Bbb F_{59})$ transitively. 
\end{lemma}

\begin{proposition} \label{calculation3''}
The pair $(G_1, G_2)$ satisfies the following two conditions:  
\begin{itemize}
\item[(a)] $G_1 \cap G_2=\{1\}$. 
\item[(b)] Let $Q \in \Bbb P^1(\Bbb F_{59})$. 
Then $G_1Q=\Bbb P^1(\Bbb F_{59})=G_2Q$. 
\end{itemize}
In particular, conditions (a) and (b) in Fact \ref{criterion1} are satisfied for the pair $(G_1, G_2)$.  
\end{proposition}

\begin{proof}[Proof of Theorem \ref{main3} (a)]
Proposition \ref{calculation3''} and Fact \ref{criterion1} imply Theorem \ref{main3} (a).  
\end{proof} 

Let $\sigma', \tau' \in {\rm Aut}(\Bbb P^1)$ be represented by matrices
$$ 
A_{\sigma'}=\left ( \begin{array}{cc}
0 & \alpha^2 \\
\alpha^{-1} & 0 
\end{array}\right), \ 
A_{\tau'}=
\left (\begin{array}{cc}
\alpha^2 & \alpha^3 \\
-1 & -1 
\end{array}\right)
$$
respectively.

\begin{lemma} \label{calculation4''}
\begin{itemize}
\item[(a)] $\sigma'$ is of order two. 
\item[(b)] $\tau'$ is of order $30$. 
\item[(c)] $A_{\sigma'}^{-1}A_{\tau'}A_{\sigma'} \sim A_{\tau'}^{-1}$. 
\item[(d)] The group $G_3:=\langle \sigma', \tau' \rangle$ is isomorphic to $D_{60}$.
\item[(e)] The group $G_3$ acts on $\Bbb P^1(\Bbb F_{59})$ transitively. 
\end{itemize}
\end{lemma}

\begin{proposition} \label{calculation5''}
The pair $(G_1, G_3)$ satisfies the following two conditions:  
\begin{itemize}
\item[(a)] $G_1 \cap G_3=\{1\}$. 
\item[(b)] Let $Q \in \Bbb P^1(\Bbb F_{59})$. 
Then $G_1Q=\Bbb P^1(\Bbb F_{59})=G_3Q$. 
\end{itemize}
In particular, conditions (a) and (b) in Fact \ref{criterion1} are satisfied for the pair $(G_1, G_3)$.  
\end{proposition}

\begin{proof}[Proof of Theorem \ref{main3} (b)]
Proposition \ref{calculation5''} and Fact \ref{criterion1} imply Theorem \ref{main3} (b).  
\end{proof}

Let $\iota$ be an automorphism of $\Bbb P^1$ represented by  
$$ 
A_{\iota}=
\left ( \begin{array}{cc}
1 & \alpha^{30} \\
0 & -\alpha^{15} 
\end{array} \right ), 
$$
and let $G_4:=\iota G_1 \iota^{-1}$. 

\begin{proposition} \label{calculation6''}
The pair $(G_1, G_4)$ satisfies the following two conditions:  
\begin{itemize}
\item[(a)] $G_1 \cap G_4=\{1\}$. 
\item[(b)] Let $Q \in \Bbb P^1(\Bbb F_{59})$. 
Then $G_1Q=\Bbb P^1(\Bbb F_{59})=G_4Q$. 
\end{itemize}
In particular, conditions (a) and (b) in Fact \ref{criterion1} are satisfied for the pair $(G_1, G_4)$.  
\end{proposition}

\begin{proof}[Proof of Theorem \ref{main3} (c)]
Proposition \ref{calculation6''} and Fact \ref{criterion1} imply Theorem \ref{main3} (c).  
\end{proof}

\end{document}